\documentclass{amsart}
\usepackage{amssymb, amsmath}
\usepackage[enableskew]{youngtab}

\usepackage{pstricks,pst-node}


\theoremstyle{plain}
\newtheorem{thm}[subsection]{Theorem}
\newtheorem{cor}[subsection]{Corollary}
\newtheorem{lem}[subsection]{Lemma}
\newtheorem{prop}[subsection]{Proposition}

\theoremstyle{definition}
\newtheorem{defn}[subsection]{Definition}
\newtheorem*{rem}{Remark}
\newtheorem*{ex}{Example}

\def\n{\mathbb{N}}

\def\z{\mathbb{Z}}
\def\ra{\rightarrow}
\def\l{\lambda}

\def\e{\epsilon}

\def\qb#1{\boldsymbol{\left[#1\right]}}
\def\ls{\hspace*{.13in}}

\def\bs{\backslash}
\def\lps{.\hspace*{-.01in}.\hspace*{-.01in}.}
\def\dlps{\raisebox{.09in}{.}\hspace*{.01in}\raisebox{.04in}{.}\hspace*{.01in}\raisebox{-.01in}{.}}
\def\vlps{\raisebox{.06in}{$\cdot$}\hspace*{-.04in}\raisebox{.01in}{$\cdot$}\hspace*{-.04in}\raisebox{-.04in}{$\cdot$}}

\def\mf#1{\mathfrak{#1}}

\numberwithin{equation}{subsection}

\newcounter{listequation}

\begin{document}

\title[Combinatorial Invariance of Relative Kazhdan-Lusztig Polynomials]{Combinatorial Invariance of Relative Kazhdan-Lusztig Polynomials in the Hermitian Symmetric Case}
\author{W. Andrew Pruett}
\begin{abstract}
We develop a marking system for an analog of Hasse diagrams of intervals $[u,v]$ with $u\leq v$ in a Hermitian symmetric pair $W/W_J$, and use this to create a closed form algorithm for computing relative $R$-polynomials. The uniform nature of this algorithm allows us to show combinatorial invariance of relative Kazhdan-Lusztig polynomials in the Hermitian symmetric space setting.
\end{abstract}

\maketitle
\section{Introduction}
\subsection{Background}
In their seminal work \cite{KL}, Kazhdan and Lusztig defined for any Coxeter group $W$ a class of polynomials indexed by $W\times W$ now known as the Kazhdan-Lusztig polynomials of $W$.  These polynomials are usually defined in terms of an auxiliary set of polynomials, the $R$-polynomials, also indexed by $W\times W$.  

In 1987, Deodhar introduced parabolic analogs of these sets of polynomials in \cite{Deodhar2}, which he referred to as the relative $R$-polynomials and relative Kazhdan-Lusztig polynomials.  The relative polynomials reduce to the standard ones in the case that a trivial parabolic subgroup of $W$ is used.  He provided a recursive calculation for these polynomials similar to the original formulations of Kazhdan and Lusztig.

In \cite{Proctor1}, Proctor obtained the Hasse diagrams of quotient Bruhat orders that constituted lattices, and in \cite{Proctor2}, he gave a combinatorial game which again yielded these lattices, but with an explicit description of the reduced decompositions of the elements of $W/W_J$ associated to Hermitian symmetric pairs.   In \cite{Jakobsen}, Jakobsen found the same diagrams as the lattices of noncompact positive roots, where he referred to these diagrams as being of Hermitian type,  a practice we will continue in this paper.  The diagrams, or at least their outlines, have since appeared in many different contexts, including \cite{LS} and \cite[Section 3]{Boe}, where they are used to define trees from which relative Kazhdan-Lusztig polynomials can be derived, \cite{Brenti1}, \cite{Brenti2}, and \cite{BrentiSym}, where lattice paths of particular types are studied, and \cite[Sections 1.2 and 7]{ThomasYong} where they are used for calculations in quantum Schubert calculus, to name a few. 

Let $\mf{g}$ be a complex simple Lie algebra, $\mf{p}$ a maximal parabolic subalgebra with Levi decomposition $\mf{p}=\mf{k}+\mf{u}$, where $\mf{k}$ is the subalgebra whose basis is the root spaces spanned by compact roots, and $\mf{u}$ the subalgebra corresponding subalgebra for the span of the noncompact roots.  It is well-known that the following conditions are equivalent:
\begin{enumerate}
\item $(\mf{g},\mf{k})$ is a (complexified) Hermitian symmetric pair,
\item $\mf{u}$ is abelian,
\item the associated flag manifold $G/P$ is cominiscule.
\end{enumerate}
This equivalence connects Jakobsen's work, which centered on the positive noncompact roots of simple complex Lie algebras, and Proctor's, which focused on the combinatorics of cominiscule flag varieties.  In this work, if $W$ is the Weyl group of $G$, and $W_J$ the Weyl group of $P$, and  $G/P$ is cominiscule, we abuse notation and refer to the quotient $W/W_J$ as a Hermitian symmetric pair.  

\subsection{This paper}
The purpose of this work is to study the relative $R$-polynomials for Hermitian symmetric pairs by means of a marking algorithm on diagrams of Hermitian type.  In particular, for $W/W_J$ a Hermitian symmetric pair, and $u<v$ in $W/W_J$ , we show that there are canonically obtained diagrams $M$ for $u$ and $\Lambda$ for $v$ such that $M$ is contained within $\Lambda$, and the corners of the skew diagram $\Lambda\bs M$ along its Northwest edges determine the relative $R$-polynomial $R_{u,v}^J(q)$ completely.  As an application of the algorithm, we prove that in this situation, $R^J_{u,v}(q)$ is dependent only on the poset isomorphism class of $[u,v]$.  In particular, if $[u,v]$ is an interval in $W/W_J$ and $[u',v']$ is an interval in $W'/W_{J'}$, then $R_{u,v}^J(q)=R_{u',v'}^{J'}(q)$.  This settles the question of combinatorial invariance for the relative $R$-polynomials for $W/W_J$ a Hermitian symmetric pair in a uniform way.  As relative Kazhdan-Lusztig polynomials are sums of relative $R$-polynomials, their combinatorial invariance follows.

It should be mentioned that a different combinatorial derivation of the $R$-polynomials has been recently published by F. Brenti (for $A_n/A_{p-1}\times A_{n-p}$, see \cite{Brenti1}, for all other types, see \cite{Brenti2}).  The algorithm we give has the benefit of being uniform for all types including $E_6/D_5$ and $E_7/E_6$.

%\subsection{}
%The organization of this paper is as follows:

\section{Preliminaries}
\subsection{Posets}
We follow \cite{Stanley} for all definitions involving posets. In particular, if $P:=(P,\leq)$ is a poset, $u,v\in P$, we write $[u,v]:=\{w:u\leq w\leq v\}$ and call this the {\it interval of $u$ and $v$}.  Similarly, $(u,v):=[u,v]\bs\{u,v\}$.  For $u\in P$, we refer to the lower interval of $u$ to be the set $\{w:w\leq u\}$. If $u<v$, and $(u,v)=\emptyset$, then we say that $u$ is covered by $v$, and write $u\triangleleft v$.  We say that $u$ and $v$ are comparable if $u\leq v$ or $v\leq u$ holds.

Recall that for $m\in\n$, 
\begin{eqnarray*}
\qb{m}&:=&\dfrac{q^m-1}{q-1}=1+q+\cdots +q^{m-1}\\
\qb{m!}&:=&\qb{m}\qb{m-1}\cdots \qb{1}
\end{eqnarray*}
where $k=m$ if $m$ is odd, and $k=m-1$ if $m$ is even.

\subsection{Coxeter groups}
We follow \cite[Chapter 2]{BB} for all Coxeter group notation and terminology. In particular, given a Coxeter system $(W,S)$, where $S$ is a set of involutions generating $W$, we can define a rank function $\ell:W\ra\n$ by $\ell(w)=$ minimal number of generators required to write $w$.  Define the left descent set $D_L(W):=\{s\in S:\ell(sw)<\ell(w)\}$, and similarly the right descent set $D_R(w)$.  For $v,w\in W$ with $v\leq w$, let $\ell(v,w):=\ell(w)-\ell(v)$.  Denote the identity of  $W$  by  $e$ . If $J\subset S$, let $W_J$, the parabolic subgroup of  $W$  generated by $J$, be the system $(\langle J\rangle,J)$, and let 
\[W/W_J:=\{w\in W:\ell(sw)>\ell(w)\text{ for all }s\in J\}.\]
We will identify $W/W_J$ with its minimal right coset representatives for the remainder of this paper, and write $W^J$ for $W/W_J$.  Then we have a well-defined length function $\ell_{W^J}(wW_J):=\ell(w)$, where $w$ is the minimal length right coset representative of $W/W_J$.  This length function makes $W^J$ into a ranked chained poset \cite[Section 2.5]{BB}.  All Coxeter groups appearing in this paper are crystallographic, and therefore are Weyl groups.

\subsection{Bruhat order}
Weyl groups bear an important order, the (strong) Bruhat order, which can be realized in several ways.  The general notion for covering in the strong Bruhat order is that $u\triangleleft_\text{Bruhat} v$ provided $\ell(u)<\ell(v)$ and there exists a reflection $t=wsw^{-1}$, $s\in S$, such that $tu=v$.  The strong Bruhat order is the transitive closure of these covering relations.  An equivalent characterization of the strong Bruhat order is the subword order.

\begin{thm}\label{subcrit}
If $u,v\in W$, then $u\leq v$ if and only if there is a reduced decomposition of $v$ such that a reduced decomposition for $u$ can be obtained by deleting $\ell(v)-\ell(u)$ simple reflections.
\end{thm}
\begin{proof}
See  \cite[Theorem 2.2.2]{BB}.
\end{proof}

A second order on Coxeter groups also exists, called the (left) weak Bruhat order.  The covering relations are given by $u\triangleleft_\text{left weak} v$ provided $\ell(u)<\ell(v)$ and there exists $s\in S$ such that $su=v$, and the weak Bruhat order is the transitive closure of these covering relations.  

Both the weak and strong Bruhat orders descend to quotients of Coxeter groups by restricting to minimal coset representatives.  It is well known that the strong and weak Bruhat order coincide provided $W^J$ is a Hermitian symmetric group, see \cite[Section 2]{Boe}.  By abuse of notation, we use $\leq$ for all Bruhat orders in this paper.

\subsection{Relative $R$-polynomials}
The object of study in this paper is a collection of polynomials introduced by Deodhar in \cite[Lemma 2.8]{Deodhar2}.
\begin{thm}[Deodhar]
Let $(W,S)$ be a Coxeter system, and $J\subseteq S$.  For each $x\in \{-1,q\}$, there is a unique set of polynomials $\{R^{J,x}_{u,v}(q)\}_{u,v\in W^J}\subset\z[q]$ such that , for all $u,v\in W^J$,
\begin{enumerate}
\item $R_{u,v}^{J,x}(q)=0$ if $u\not\leq v$,
\item $R_{u,u}^{J,x}(q)=1$,
\item if $s\in D_L(v)$ and $u<v$ 
\[R_{u,v}^{J,x}(q)=\begin{cases}R_{su,sv}^{J,x}(q)&\text{ if }su<u\\(q-1-x)R_{u,sv}^{J,x}(q)&\text{ if }u<su\notin W^J\\(q-1)R_{u,sv}^{J,x}(q)+qR_{su,sv}^{J,x}(q)&\text{ if }u<su\in W^J\end{cases}\]
\end{enumerate}
\end{thm}

The polynomials from the above theorem are the parabolic (relative) $R$-polynomials of $W^J$ of type $x$. If $J=\emptyset$, it is trivial to verify that $R^{\emptyset,-1}_{u,v}(q)=R_{u,v}^{\emptyset,q}=R_{u,v}(q)$, where $R_{u,v}(q)$ is the standard $R$-polynomial as defined in \cite[Theorem 5.1.1]{BB}.  

In this paper, we will be working only with the polynomials for which $x=-1$.

\section{Essential definitions}\label{Sdiagrams}
\subsection{Diagrams of Hermitian type}
\begin{defn}\label{HSDiag}
We refer to the following diagrams as the diagrams of Hermitian type:
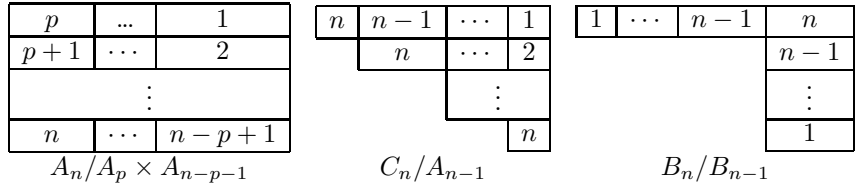
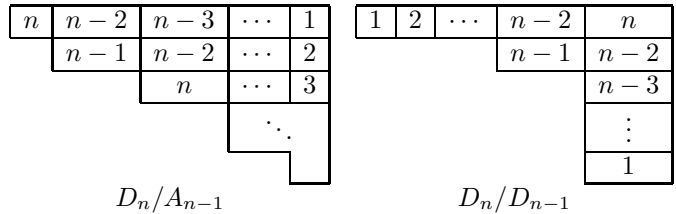
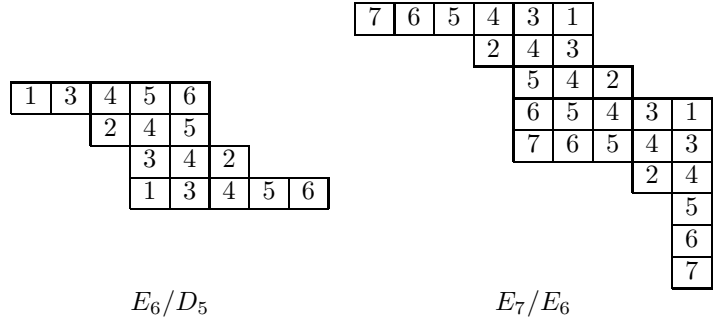
\[\begin{array}{ccc}
%An
\begin{array}{|c|c|c|}\hline p&\lps&1\\\hline p+1&\cdots&2\\\hline\multicolumn{3}{|c|} \vdots\\\hline n&\cdots&n-p+1\\\hline \end{array}&
%Bn/An-1
\begin{array}{|c|c|c|c|}\hline n&n-1&\cdots&1\\\cline{1-4}\multicolumn{1}{c|}{}&n&\cdots &2\\\cline{2-4}\multicolumn{2}{c|}{}&\multicolumn{2}{|c|}{\vdots}\\\cline{3
-4}\multicolumn{3}{c|}{}&n\\\cline{4-4}\end{array}&
%Bn/Bn-1
\begin{array}{|c|c|c|c|}\cline{1-4} 1&\cdots&n-1&n\\\hline\multicolumn{3}{c|}{}&n-1\\\cline{4-4}\multicolumn{3}{c|}{}&\vdots\\\cline{4-4}\multicolumn{3}{c|}{}& 1\\\cline{4-4}\end{array}\\
\vspace*{.5in}A_n/A_p\times A_{n-p-1}&C_n/A_{n-1}&B_n/B_{n-1}\\
\end{array}\]

\[\begin{array}{cc}
\begin{array}{|c|c|c|c|c|}\hline n&n-2&n-3&\cdots &1\\\hline\multicolumn{1}{c|}{}&n-1&n-2&\cdots&2\\\cline{2-5}\multicolumn{2}{c|}{}&n&\cdots&3\\\cline{3-5}\multicolumn{3}{c|}{}&\multicolumn{2}{|c|}{\ddots}\\\cline{4-4}\multicolumn{4}{c|}{}&\\\cline{5-5}\end{array}&
\begin{array}{|c|c|c|c|c|}\hline 1&2&\cdots&n-2&n\\\hline\multicolumn{3}{c|}{}&n-1&n-2\\\cline{4-5}\multicolumn{4}{c|}{}&n-3\\\cline{5-5}\multicolumn{4}{c|}{}&\vdots\\\cline{5-5}\multicolumn{4}{c|}{}&1\\\cline{5-5}\end{array}\\
D_n/A_{n-1}&D_n/D_{n-1}
\end{array}\]

\[\begin{array}{cc}
\begin{array}{|c|c|c|c|c|c|c|c|}\cline{1-5}1&3&4&5&6&\multicolumn{3}{|c}{}\\\cline{1-5}\multicolumn{2}{c|}{}&2&4&5&\multicolumn{3}{|c}{}\\\cline{3-6}\multicolumn{3}{c|}{}&3&4&2&\multicolumn{2}{|c}{}\\\cline{4-8}\multicolumn{3}{c|}{}&1&3&4&5&6\\\cline{4-8}\end{array}&
\begin{array}{|c|c|c|c|c|c|c|c|c|}\cline{1-6}7&6&5&4&3&1&\multicolumn{3}{|c}{}\\\cline{1-6}\multicolumn{3}{c|}{}&2&4&3&\multicolumn{3}{|c}{}\\\cline{4-7}\multicolumn{4}{c|}{}&5&4&2&\multicolumn{2}{|c}{}\\\cline{5-9}\multicolumn{4}{c|}{}&6&5&4&3&1\\\cline{5-9}\multicolumn{4}{c|}{}&7&6&5&4&3\\\cline{5-9}\multicolumn{7}{c|}{}&2&4\\\cline{8-9}\multicolumn{8}{c|}{}&5\\\cline{9-9}\multicolumn{8}{c|}{}&6\\\cline{9-9}\multicolumn{8}{c|}{}&7\\\cline{9-9}\end{array}\\
E_6/D_5&E_7/E_6
\end{array}
\]
\end{defn}

Turning a diagram of any above type counterclockwise by $135^\circ$ yields a lattice that is the Hasse diagram of the Bruhat order of the specified quotient.  This fact was shown by Proctor, and the following is essentially a restatement of \cite{Proctor1}, Proposition 3.2.

\begin{prop}[Proctor]
The subdiagrams of $W/{W}_J$ correspond bijectively to the elements of $W/{W}_J$, for each pair $W/W_J$ listed above.
\end{prop}

Proctor provided a different proof of this result in \cite[Theorem A]{Proctor2} using a modification of a numbers game of Mozes.  Alternatively, we could show the result using a sorting game that is essentially dual to Proctor's method.  We give the details in $A_n/A_{p-1}\times A_{n-p}$ in the following section.

These Hasse diagrams appeared almost simultaneously in \cite[Section 4]{Jakobsen} as the noncompact root lattices associated with unitarizable highest weight modules. 

\subsection{Diagrams and partitions}
The following definition is fundamental to this paper.

\begin{defn}\label{diagrams}
A {\it subdiagram} of one of the above Hermitian type diagrams is a subset of the given diagram such that, if a box $b$ is in the subdiagram, every box located above and left of $b$ is also in the subdiagram.  \end{defn}

Importantly, diagrams retain the labels, which we think of as corresponding to simple reflections in the standard (Bourbaki) ordering of the generators $S$ of the Weyl group $W$, imagining $J\subset W'$ as laying in $S$ in the canonical way.

\begin{defn}
A {\it partition } is a diagram with its labels removed.  We refer to the partition $\l$ underlying a diagram $\Lambda$ as the shape of $\Lambda$.
\end{defn}

Then the subdiagrams of the above can be interpreted in the following way:  each $w\in W^J$ corresponds to a partition of $\ell(w)$ of type $W^J$ with the following properties:\\

\begin{tabular}{|c|c|}\hline
Quotient&partition description\\\hline
$A_n/A_{p-1}\times A_{n-p}$ & $\l=(\l_1,\dots,\l_k)$, $k\leq n+1-p$, $\l_i\leq p$,\\
& where $s_p$ is the unique simple reflection of  $s$  not contained in $J$.\\\hline
$C_n/A_{n-1}$ & $\l=(\l_1,\dots,\l_m)$, $n\geq \l_1>\l_2\cdots>\l_m> 0$.\\\hline
$B_n/B_{n-1}$ &  $\l=(m)$, $m\leq n+1$, or $\l=(n+1,1,1,\dots, 1)$.\\\hline
$D_n/A_{n-1}$: & $\l=(\l_1,\dots,\l_m)$, $n\geq \l_1>\l_2\cdots>\l_m> 0$.\\\hline
$D_n/D_{n-1}$: & $ \l=(m)$, $m\leq n$, or $\l=(n,m)$, $m\in\{1,2\}$ or $\l=(n,2,1,1,\dots,1)$.\\\hline
$E_6/D_5, E_7/E_6$ & irregular, best considered individually.\\\hline
\end{tabular}\\\\
Furthermore, because we require subdiagrams to be upper left justified, knowing a partition and its type is enough to construct the diagram for that partition.  

In this light, it makes sense to make the following definition:

\begin{defn}
The {\it partition of a permutation $w$ } is the partition underlying the diagram associated to $w$ through the bijection detailed in the following section.
\end{defn}

The next definitions will only be used in Section \ref{CombInv}, but we state them now for completeness.

\begin{defn}
We say that a partition is {\it standard } if it is a Young tableau under the usual definition (i.e. consists of a top and left justified set of rows, weakly decreasing in length).  We say that a partition is {\it standard skew } if it can be written as a skew tableau of two standard tableaus.  Note that all standard tableaus $\mu$ are standard skew, realized as $\mu\bs\emptyset$.  
\end{defn}

\begin{defn}
We say that $W^J$ {\it supports a diagram of shape $\l/\mu$ } if $\l\bs \mu$ corresponds to $\Lambda\bs M$, with $\Lambda$ the diagram of $v$ in $W^J$, and $M$ the diagram of $u$. Given $u<v$ in $W^J$ with diagrams $M$ and $\Lambda$, respectively, we refer to $\Lambda\bs M$ as a skew diagram of type $W^J$, or of type $W$, when $J$ is clear from context.  The support for the skew diagram for $v\bs u$ where $u$ and $v$ are specific elements of some quotient $W^J$ is the set of simple reflections that appear in the filling of $v\bs u$.
\end{defn}

\section{The sorting game}
\subsection{The idea of the algorithm}
An alternate realization of the diagrams of Hermitian type can be had by modifying a sorting game developed by Eriksson and Eriksson \cite[Section 9]{EE}.  We begin by noting that each Weyl group can be associated to a permutation group:  $A_n$ with the symmetric group on $n+1$ letters, etc.  We realize each of the above quotients in their permutation representation, using the permutation representation described in \cite[Section 6.2]{HE} for $E_6/D_5$ and $E_7/E_6$ and the standard representations described in \cite[Chapter 8]{BB} for the others, then use a sorting argument to construct tableau.  The rows in the constructed tableau correspond to rows in the above diagrams in a natural way, and because we are working in Hermitian symmetric spaces, the Bruhat and weak Bruhat orders coincide on these quotients, allowing us to read the left descent sets of each element from the diagram (see Proposition \ref{ordertheorem}).  Furthermore, because the sorting game is modeling multiplication in the Weyl group, the diagrams we derive can be associated to reduced decompositions.  This will be a central tool in the statement and proof of the main theorem.  

We present a single case of the sorting game.  The rest are very similar.

\subsection{Case: $A_n/A_{p-1}\times A_{n-p}$}
Suppose $W=A_n$ with the standard generating set $S$ of adjacent transpositions, and $J=S\bs\{s_p\}$.  Then $W^J$ consists of permutations increasing in positions 1 through $p$ (called segment one), and in positions $p+1$ through $n+1$ (called segment two).  Write $w(j)$ for the image of $w$ under the action of the permutation $w$; if $w$ is written in complete notation, we see this is just the $j^\text{th}$ entry in $w$.  Define a game on $w\in W^J$ in the following way:  in step one, sort $w(p+1)$ into the segment $[w(1),\dots,w(p)]$, and record the number of elements "jumped" as $d_1$.  Call the resulting permutation $w_1$, and sort $w_1(p+2)=w(p+2)$ into the segment $[w_1(1),\dots,w_1(p+1)]$, again recording the number of jumped elements as $d_2$.  as $w(p+1)<w(p+2)$ by assumption, $d_2\leq d_1$.  Continue until the permutation is completely sorted, leaving a partition $(d_1,d_2,\dots,d_k)$.  Note that $0\leq k\leq n+1-p$, and that $k=0$ if and only if  $w$  was sorted before any action was taken (which implies that  $w$  is the identity).  

Each sort can be realized as multiplication by a specific product of simple reflections; in particular, the action of sorting position $p+k+1$ to position $p-j$ is $s_{p+k}s_{p+k-1}\cdots s_{p-j}$, corresponding to the diagram row word
\[r(p,\l_k):=\begin{array}{|c|c|c|c|}\hline s_{p+k}&\cdots& s_{p-j+1}&s_{p-j}\\\hline\end{array}.\]
Thus we can associate to each partition an element of $W^J$ by 
\[w=r(p,\l_k)^{-1}r(p,\l_{k-1})^{-1}\cdots r(p,\l_1)^{-1},\] 
and so we have a map
\[\phi:A_n/(A_{p-1}\times A_{n-p})\mapsto\{\l=(\l_1,\dots,\l_k),|\l|=\ell(w), \l_1\leq p, k\leq n+1-p\}\]
where the labeling is exactly as in Section \ref{HSDiag}.  This map is easily shown to be bijective.

\begin{rem} 
All further cases are similar, and so we give only the map 
\[\phi:W^J\mapsto\{\text{partitions with particular conditions}\}.\]
\end{rem}

\subsection{$C_n/A_{n-1}$}  Here,  $W$  can be viewed as Weyl group of type $C_n$, i.e. permutations of $\{-n,\dots,-1,1,\dots,n\}$ with $w(i)=-w(-i)$, and where we take $J=A_{n-1}$.  We see that $W^J:=\{w\in C_n:\exists!1\leq p\leq n\text{ with }w(p)>0,w(p)<0,\text{ and }w(1)<\cdots< w(p),w(p+1)<\cdots w(n)<0\}$.  Note that $w(n)>0$ if and only if  $w$  is the identity in this case.  The algorithm is as follows:  set $w'(n)=-w(n)$, and sort it into $[w(1),w(2),\dots,w(n-1)]$, recording $d_1$ to be one more than the number of jumped elements, and call the resulting permutation $w_1$.  Continue until $w_k(n)>0$.  Then $\phi(w)=(d_1,d_2,\dots, d_k)$.  As we assumed that $w(i)<w(j)<0$, and $w(i)$ is sorted before $w(j)$, then $w(i)$ is moved more positions to the left than $w(j)$, showing that $d_j<d_i$ for $j>i$.  Thus we have a map 
\[\phi:C_n/A_{n-1}\mapsto\{\l=(\l_1,\dots,\l_k),|\l|=\ell(W), \l_i<\l_j\text{ if }i<j, \l_1\leq n\}\]

Using a row word argument with the labels fixed as before, this map is bijective.

\subsection{$B_n/B_{n-1}$:}  Take $W=B_n$, and $J$ the standard generating set of $B_{n-1}$ chosen such that it lies within the standard generating set for $B_n$, with the same enumeration.  Recall that as permutation groups, $B_n$ and $C_n$ are isomorphic, so $w\in B_n$ is a permutation of $[-n,\cdots,-1,1,\cdots,n]$ such that $w(i)=-w(-i)$ for $1\leq i\leq n$.  Thus we can write $w$ as $[w(1),\dots, w(n)]$, and $W^J=\{w\in B_n:0<w(2)<\cdots<w(n)\}$.  The only negative element, if there is one, is in position $1$.  The sort has two parts:  first, if $w(1)>0$, sort $w(1)$ into the segment $\{w(2),\dots,w(n)\}$, recording the number of right jumps as $d_1$, and the algorithm terminates. If $w(1)<0$, let $w_1$ be $(w(2),\dots, w(n),-w(1))$, recording $d_1=n$.  Next, if necessary, sort $\{w_1(n-1),w_1(n)\}$, recording the number of left jumps of $w_1(2)$ as $d_2$.  Continue sorting larger terminal segments in increments of one; note that as we increase the size, we can at most move an entry one position, thus $d_k\in\{0,1\}$ for $k>1$.  The algorithm terminates when $d_k=0$.

\subsection{$D_n/A_{n-1}$}  Here  $W$  is Weyl group of type $D_n$, i.e. permutations of $-[n]\cup[n]$ with an even number of negative entries in positions 1 through  $n$, and with $w(i)=-w(-i)$ for $1\leq i\leq n$.  Again we write $w\in D_n$ as $[w(1),\dots, w(n)]$. It is easy to see that $W^J:=\{w\in D_n:\exists!1\leq p\leq n\text{ with }w(p)>0,w(p)<0,\text{ and }w(1)<\cdots< w(p),w(p+1)<\cdots w(n)<0\}$.  Given a permutation  $w$ , we have a two-step method of sorting:  first, let $w'(n)=-w(n-1)$, and $w'(n-1)=-w(n)$, and begin by sorting $w'(n-1)$ into $\{w'(1),w'(2),w(3),\dots, w(n)\}$, noting $d_1$ as one more than the number of jumped elements, and denote the resulting permutation as $w''$ (note that $w''\notin W^J$).  Then sort $w'(n)$ into $\{w"(1),w''(2)\cdots,w''(n-1)\}$, letting $d_2$ be the number of jumped elements.  Again, by our assumptions $d_2<d_1$ (as $w'(n)<w'(n-1)$, implying it is sorted fewer steps right).  Continue the same process with the resulting permutation.  Note that $d_{2k}$ may be zero for some $k$, and this may happen only if there are no remaining negative entries.  The process halts when no negative entries remain at the beginning of an odd step.  This gives a map  
\[\phi:D_n/A_{n-1}\mapsto\{\l=(\l_1,\dots,\l_k):|\l|=\ell(w), \l_i<\l_j\text{ if }i<j, \l_1\leq n-1\}.\]

\subsection{$D_n/D_{n-1}$}  In this case, the Bourbaki ordering of simple reflections admits no ``nice" combinatorial interpretation in terms of the game we have presented so far. Instead, we present the proof using the ordering given in \cite{BB}.  Explicitly, if we denote the Bj\"orner and Brenti enumeration of generators $b_i$, $0\leq i\leq n-1$, and the Bourbaki ordering $s_i$, $1\leq i\leq n$, we have that $b_i=s_i$ for $1\leq i\leq n-1$, and $b_0$ acts as the cycles $(-1,2)(1,-2)$, where $s_n$ acts as the cycles $(n-1,-n)(n,-n+1)$.  This is effectively a diagram bijection carrying $s_i\leftrightarrow s_{n-i}$.  The difference is purely cosmetic, but it allows an easier statement of the sorting algorithm.

Take $W=D_n$, and $J$ the generating set of $D_{n-1}$ chosen such that it lies within the standard generating set for $D_n$, with the same enumeration.  Then $W^J=\{w\in D_n:w(-2)<w(1)<\cdots<w(n-1)\}$.  Again, the sorting process has several steps:  begin by sorting $w(n)$ into $\{w(1),\dots, w(n-1)\}$.  We observe that this step is trivial only in the case $w=e$, as otherwise $w(n)<w(n-1)$ as is easily verified from the definition of the quotient.  Explicitly, $w(-2)<w(1)<w(2)$ implies that if any negatives appear in positions 1 through  $n$  of  $w$ , than they must be in positions $1$ and  $n$ , implying that $w(n)<0<w(n-1)$; if no negatives appear, then $\sigma_1=\{w(1),\dots, w(n-1)\}$ represents an ascending sequence taken from $[n]$, and it is only the trivial sequence that fails to utilize  $n$  in $\sigma_1$, showing $w(n)<n=w(n-1)$.  Record the number of elements jumped by $w(n)$ as it moves left as $d_1$.  

If there are no negative entries, then the resulting sequence is sorted, and the process terminates.  Otherwise, denote the resulting permutation as $w_1$, and let $w'(1)=-w_1(2)$, $w'(2)=-w_1(1)$.  

Sort $w_1(3)$ into the segment $\{w'(1),w'(2)\}$ recording one more than the number of jumped elements as $d_2$ (note that it is possible that no jump occurs). We claim that $1\leq d_2\leq 2$:  to see this, note first that $w(-2)<w(1)<w(2)$ implies $w(-2)<w(-1)=-w(1)<w(2)$, and also that, as we must assume $w(n)<0$ to have reached this step, we know that the sets $\{w'(1),w'(2)\}=\{-w(1),-w(n)\}$.  Hence sorting $w(2)$ into $\{w'(1),w'(2)\}$, we can't pass $-w(1)$, meaning we can jump at most one element.

Denote the resulting permutation as $w_2$, and jump $w_2(4)$ into $\{w_2(1),w_2(2),w_2(3)\}$, recording $d_3$ as the number of left jumps; again, and for the same reason, $d_3\leq 0$.  Continue until no left jump is possible, at which point in time the sequence is sorted.  The partition is $(d_1,d_2,\dots, d_k)$, and labeling as in the figure realize the sorts involved, giving the bijection as before.

\subsection{$E_6/D_5$ and $E_7/E_6$}
These cases are best checked by hand, using the permutation descriptions of \cite{HE}/

\begin{rem} 
The row words that appear in each case give the labeling of the diagrams in \ref{HSDiag}.  
\end{rem}

\section{Order and descents}
\subsection{Young's lattice}
Recall Young's lattice on partitions: two partitions $\mu=(\mu_1,\dots,\mu_r)$ and $\l=(\l_1,\dots, \l_s)$ are related by $\mu<\l$ provided $r\leq s$, and that for $1\leq i\leq r$ $\mu_i\leq \l_i$.  Visually, if the partitions are represented as Young diagrams, this just says that the diagram of $\mu$ sits inside that of $\l$ when their upper left corners coincide.  

In the current context, if $M$ and $\Lambda$ are diagrams of type $W^J$ with shapes $\mu:=(\mu_1,\dots,\mu_r)$ and $\lambda:=(\l_1,\dots,\l_s)$ respectively, then the labels in their upper left corners again coincide, and we again say that $M\leq_\text{diagram}\Lambda$ provided $\mu_i\leq\l_i$ for each $1\leq i\leq r\leq s$.  We call this the containment order on $W^J$, or just the containment order when the context is clear.  It is a consequence of the definition of subdiagrams that $W^J$ is a poset with respect to this ordering.

\subsection{Diagrams and order}
Given a diagram $\Lambda$ of type $W^J$, we can visually distinguish the ascents and descents of $w\in W^J$ associated to $\Lambda$.

\begin{prop}\label{ordertheorem}
Let $u\in W^J$ with diagram $\Lambda$, and suppose that $su\in W^J$ with diagram $M$.  Then
$s\in D_L(u)$ if and only if the skew diagram $\Lambda\bs M$ consists of a single box, positioned in an outside corner of $\Lambda$.  Similarly, if $s\notin D_L(u)$, then $M\bs\Lambda$ is a single box that occupies an outside corner of $M$.
\end{prop}
\begin{proof}

If $s$ is a left descent of $u\in W^J$, then as the weak Bruhat and strong Bruhat orders coincide for Hermitian symmetric pairs, we can write any reduced decomposition of $u$ as $asb$, with $sa=as$, with $s$ not appearing in $a$.  Then in the lexicographically minimal reduced form, we can represent $u$ as $asb$ as above, and $a$ can contain no copy of $s$, and $as=sa$.  but this says no $s'$ labeled box appears in $a$, which means no $s'$ labeled box appear in the partition for $u$ after the final occurrence of an $s$-labeled box, i.e. that $s$-labeled box is an outside corner of $\Lambda$.

Conversely, if $s$ is the label of an outside corner of $\Lambda$, say on row $j$, then $s$ commutes with all labels of boxes that appear on rows lower than $j$, and nothing appears right of the last $s$-labeled box on row $j$.  Thus $s$ commutes with all simple reflections that appear left of the first occurrence of $s$ in the lexicographically minimal reduced decomposition, implying $s$ is a left descent.

The second statement is similar.
\end{proof}

\subsection{A poset isomorphism}
The following follows easily from Proposition \ref{ordertheorem}.
\begin{cor}\label{shape}
The poset $(W^J,\leq_\text{diagram})$ is poset isomorphic to $(W^J,\leq_\text{left weak})$.  Thus $u\leq v$ in $W^J$ if and only if the diagram of $u$ is contained in the diagram of $v$.
\end{cor}
\begin{proof}
Covering relations in the left weak order are determined precisely by left simple descents.  By the above,  $s$ is a left descent of $v$ if and only if an $s$-labeled box is an outside corner in the diagram of $v$.  

If $u$ covers $v$ in the left weak order, then there exists $s\in D_L(v)$ such that $su=v$.  By Proposition \ref{ordertheorem}, there is an $s$-labelled outside corner in the diagram $\Lambda$ of $v$ that is not contained in the diagram $M$ of $u$.  Then the $s$-labeled box can be removed from $\Lambda$, leaving a subdiagram, and considering the diagrams as products of row words forming a reduced decomposition, this corresponds to left multiplication by $s$.  Thus the resulting diagram is $M$.  

Similarly, if $M\subset \Lambda$, with $M$ the diagram of $u$ and $\Lambda$ the diagram of $v$, then this presentation by diagrams witnesses that $u$ is a subword of $v$, and so $u<v$ by the Subword Criterion (Theorem \ref{subcrit}).
\end{proof}

\section{The main theorem}\label{marks}
\subsection{Marking the diagrams}
Given $u\leq v$ in $W^J$, let $\Lambda$ be the diagram of $v$, and $M$ the diagram of $u$.  As $u\leq v$, then $M\subset \Lambda$ by Corollary \ref{shape}.  By Proposition \ref{ordertheorem}, $s\in D_L(u)$ if and only if an $s$-labeled box is an outside corner of the diagram of $u$, and $su$ is a minimal coset representative covering $u$ if and only if we can append an $s$-labeled box to $M$ in such a way that the resulting diagram is a subdiagram in the sense of Definition \ref{diagrams}. 

For an $s$-labeled box in $\Lambda\bs M$, define
\[\delta(b):=\#\{s\text{-labeled boxes in $\Lambda\bs M$ above and left of $b$, including $b$}\}.\]
We define a box $b$ to be even provided $\delta(b)$ is even, and similarly for odd.

\begin{enumerate}
\item If $s\notin D_L(u)$, $su\in W^J$, and $s$ is a short (long) reflection, then mark every (odd) $s$-labeled box with a $+$.
\item If $s\in D_L(u)$, and $s$ is a short (long) reflection, then mark every (even) $s$-labeled box in $\Lambda\bs M$ with a $-$. 
\end{enumerate}

For each box $b$ in $\Lambda\bs M$, define
\[\Delta(b):=\begin{cases}\#\{\text{boxes in $\Lambda\bs M$ above and left of $b$ , including $b$}\}&\text{if $b$ marked}\\1&\text{if otherwise}\end{cases}.\]

Suppose $u<v$, and that the diagram of $v$ is $\Lambda$, and the diagram of $u$ is $M$.  Then we define $k_{u,v}:=\#\{+\text{ marked boxes in }\Lambda\bs M\}-\#\{-\text{ marked boxes}\}$.

\subsection{Statement of Main theorem}
Now we can state the main result of this paper.

\begin{thm}[Main Theorem]
Given $u\leq v$ in $W^J$ of rank $n$, let $\Lambda$ be the diagram of $v$, and $M$ the diagram for $u$, and mark the skew diagram $\Lambda\bs M$ as in Section \ref{marks}.  Define
\begin{equation}\label{MT}
\tilde R_{u,v}^J(q)=
\begin{cases}
1&\text{ if }u=v\\
0&\text{ if }u\not\leq v\\
q^{\eta_{u,v}}(q-1)^{k_{u,v}}\prod_{b\in\Lambda\bs M}\qb{\Delta(b)}^a&\text{ if }u<v,\end{cases}\end{equation}
where the second product should be interpreted as a product over all boxes in the marked diagram $\Lambda/M$, and $a$ is $+1$ for $+$ marked boxes, $-1$ for $-$ marked boxes, with $\eta(u,v)$ is the unique natural number so that the degree of the right hand side is $\ell(v)-\ell(u)$.  Then $R_{u,v}^J(q)=\tilde R_{u,v}^J(q)$.
\end{thm}

For ease of discussion, let $\tilde r_{u,v}^J(q)$ be the function defined as $R_{u,v}^J(q)$, but with $\eta_{u,v}(s)=0$ for all $s$.  We provide several examples to aid in the discussion:

\subsection{Examples}
\begin{ex}
Consider $u=\left((543)(65)\right)^{-1}$, $v=\left((54321)(6543)(7654)(8765)(9)\right)^{-1}$ in $A_{10}/A_4\times A_5$.  Again, we use the convention of using $i$ for $s_i$.  Then

\begin{eqnarray*}
&u=\begin{array}{|c|c|c|}\hline 5&4&3\\\hline 6&5&\multicolumn{1}{c}{}\\\cline{1-2}\end{array}\qquad 
v=\begin{array}{|c|c|c|c|c|}
\hline 5&4&3&2&1\\\hline 6&5&4&3\\\cline{1-4} 7&6&5&4&\multicolumn{1}{c}{}\\\cline{1-4} 8&7&6&5&\multicolumn{1}{c}{}\\\cline{1-4} 9&\multicolumn{4}{c}{}\\\cline{1-1}\end{array}&\\\\
&\Lambda\bs M \quad=\quad\begin{array}{|c|c|c|c|c|}
\hline X&X&X&\hspace*{.13in}&\hspace*{.13in}\\\hline X&X&&\ls\\\cline{1-4} &&&&\multicolumn{1}{c}{}\\\cline{1-4} &&&&\multicolumn{1}{c}{}\\\cline{1-4} &\multicolumn{4}{c}{}\\\cline{1-1}\end{array}\quad=\quad
\begin{array}{|c|c|c|c|c|}
\hline X&X&X&+&\hspace*{.13in}\\\hline X&X&+&-\\\cline{1-4} +&&-&+&\multicolumn{1}{c}{}\\\cline{1-4} &+&&-&\multicolumn{1}{c}{}\\\cline{1-4} &\multicolumn{4}{c}{}\\\cline{1-1}\end{array}&\\
\end{eqnarray*}
and $\tilde R_{u,v}(q)=(q-1)^{5-3}\dfrac{(1)(1+q)(1)(1+q)(1)}{(1)(1+q)(1)}q^{10}=(q-1)^2q^{10}(1+q)$.
\end{ex}

\begin{ex}
Let $u=\left((86543)(7)\right)^{-1}$ and $v=\left((8654321)(765432)(86543)(7654)(865)(7)\right)^{-1}$ in $D_8/A_7$.  Then we have
\begin{eqnarray*}
&u=\begin{array}{|c|c|c|c|c|}\hline 8&6&5&4&3\\\hline\multicolumn{1}{c|}{}&7\\\cline{2-2}\end{array}\qquad 
v=\begin{array}{|c|c|c|c|c|c|c|}\hline 8&6&5&4&3&2&1\\\hline\multicolumn{1}{c|}{}&7&6&5&4&3&2\\\cline{2-7}\multicolumn{2}{c|}{}&8&6&5&4&3\\\cline{3-7}\multicolumn{3}{c|}{}&7&6&5&4\\\cline{4-7}\multicolumn{4}{c|}{}&8&6&5\\\cline{5-7}\multicolumn{5}{c|}{}&7\\\cline{6-6}\end{array}&\\\\
&\Lambda\bs M \quad=\quad  \begin{array}{|c|c|c|c|c|c|c|}\hline X&X&X&X&X& \hspace*{.13in}& \hspace*{.13in}\\\hline\multicolumn{1}{c|}{}&X&&&&&\\\cline{2-7}\multicolumn{2}{c|}{}&&&&&\\\cline{3-7}\multicolumn{3}{c|}{}&&&&\\\cline{4-7}\multicolumn{4}{c|}{}&&&\\\cline{5-7}\multicolumn{5}{c|}{}&\\\cline{6-6}\end{array}
\quad=\quad\begin{array}{|c|c|c|c|c|c|c|}\hline X&X&X&X&X& +& \hspace*{.13in}\\\hline\multicolumn{1}{c|}{}&X&+&&&-&+\\\cline{2-7}\multicolumn{2}{c|}{}&&+&&\hspace*{.13in}&-\\\cline{3-7}\multicolumn{3}{c|}{}&-&+&&\\\cline{4-7}\multicolumn{4}{c|}{}&&+&\\\cline{5-7}\multicolumn{5}{c|}{}&-\\\cline{6-6}\end{array}&\\
\end{eqnarray*}
and $\tilde R_{u,v}(q)=(q-1)^{6-4}\dfrac{\qb{4!}\qb{2!}}{\qb{2}\qb{4}\qb{2!}}q^{16}=(q-1)^2(1+q+q^2)q^{16}$.
\end{ex}

\begin{ex}
Take $u=\left((13456)(245)\right)^{-1}$ and $v=\left((13456)(245)(342)(1345)\right)^{-1}$ in $E_6/D_5$.  Then
\begin{eqnarray*}
&u=\begin{array}{|c|c|c|c|c|}\hline1&3&4&5&6\\\hline\multicolumn{2}{c|}{}&2&4&5\\\cline{3-5}\end{array}\qquad 
v=\begin{array}{|c|c|c|c|c|c|c|}\cline{1-5}1&3&4&5&6\\\cline{1-5}\multicolumn{2}{c|}{}&2&4&5\\\cline{3-6}\multicolumn{3}{c|}{}&3&4&2\\\cline{4-7}\multicolumn{3}{c|}{}&1&3&4&5\\\cline{4-7}\end{array}&\\\\
&\Lambda\bs M \quad=\quad\begin{array}{|c|c|c|c|c|c|c|}\cline{1-5}X&X&X&X&X\\\cline{1-5}\multicolumn{2}{c|}{}&X&X&X\\\cline{3-6}\multicolumn{3}{c|}{}&&&\hspace*{.13in}\\\cline{4-7}\multicolumn{3}{c|}{}&&&&\hspace*{.13in}\\\cline{4-7}\end{array}\quad=\quad
\begin{array}{|c|c|c|c|c|c|c|}\cline{1-5}X&X&X&X&X\\\cline{1-5}\multicolumn{2}{c|}{}&X&X&X\\\cline{3-6}\multicolumn{3}{c|}{}&+&&\hspace*{.13in}\\\cline{4-7}\multicolumn{3}{c|}{}&&+&&-\\\cline{4-7}\end{array}&\\
\end{eqnarray*}
and $\tilde R_{u,v}(q)=(q-1)^{2-1}\dfrac{\qb{1}\qb{2}}{\qb{2}}q^6=(q-1)q^6$.
\end{ex}

\begin{ex}
Let $u=\left((13456)(2)\right)^{-1}$ and $v=\left((13456)(245)(342)(13456)\right)^{-1}$ in $E_6/D_5$.  Then
\begin{eqnarray*}
&u=\begin{array}{|c|c|c|c|c|}\hline1&3&4&5&6\\\hline\multicolumn{2}{c|}{}&2\\\cline{3-3}\end{array}\qquad 
v=\begin{array}{|c|c|c|c|c|c|c|c|}\cline{1-5}1&3&4&5&6\\\cline{1-5}\multicolumn{2}{c|}{}&2&4&5\\\cline{3-6}\multicolumn{3}{c|}{}&3&4&2\\\cline{4-8}\multicolumn{3}{c|}{}&1&3&4&5&6\\\cline{4-8}\end{array}&\\\\
&\Lambda\bs M \quad=\quad\begin{array}{|c|c|c|c|c|c|c|c|}\cline{1-5}X&X&X&X&X\\\cline{1-5}\multicolumn{2}{c|}{}&X&&\\\cline{3-6}\multicolumn{3}{c|}{}&&\hspace*{.13in}&\hspace*{.13in}\\\cline{4-8}\multicolumn{3}{c|}{}&&\hspace*{.13in}&\hspace*{.13in}&\hspace*{.13in}&\hspace*{.13in}\\\cline{4-8}\end{array}\quad=\quad
\begin{array}{|c|c|c|c|c|c|c|c|}\cline{1-5}X&X&X&X&X\\\cline{1-5}\multicolumn{2}{c|}{}&X&+&\\\cline{3-6}\multicolumn{3}{c|}{}&&+&-\\\cline{4-8}\multicolumn{3}{c|}{}&&\hspace*{.13in}&+&\hspace*{.13in}&-\\\cline{4-8}\end{array}
&\\
\end{eqnarray*}
with $\tilde R_{u,v}(q)=(q-1)^{3-2}\dfrac{\qb{3!}}{\qb{2}\qb{3}}q^9=(q-1)q^9$.
\end{ex}

\section{Strategy of the proof, and a trivial case}
\subsection{Strategy}
We will show that $\tilde R_{u,v}^J(q)$ obeys the recursions required by the relative $R$-polynomials.  One case of the recursion is trivial, and will be treated case independently in Proposition \ref{triv}.  The strategy for the proof of the other recursive rules is as follows: we identify the possible diagrammatic presentations of each rule, and reduce to a trivial identity about quantum integers.  The discussion in Section \ref{restriction} will detail how we can restrict the skew diagram associated to the pair $u<v$ to a smaller diagram and prove the recursion there.  

\subsection{A simple case of the recurrence}
We can show that the $\tilde R$-polynomials satisfy the second recursive relation in Deodhar's definition of relative $R$-polynomials without appealing to any particular Hermitian symmetric pair.
\begin{prop}\label{triv}
Let $u,v\in W^J$ with $u<v$, and suppose that $s\in D_L(v)\bs D_L(u)$, and that $su\notin W^J$.  Then $\tilde R_{u,v}^J(q)=q\tilde R_{u,sv}^J(q)$.
\end{prop}
\begin{proof}
If $\Lambda$ and $M$ are the diagrams for $v$ and $u$ respectively, the situation in the proposition is that there is an $s$-labeled outside corner of $v$ that is neither an inside or outside corner of $u$.  Thus the $s$-labeled boxes in $\Lambda\bs M$ are not marked with $+$ or $-$.  Similarly, if $\Lambda'$ is the diagram for $sv$, the only difference between $\Lambda$ and $\Lambda'$ is a single unmarked box, implying that $\tilde r_{u,v}^J(q)=\tilde r_{u,sv}^J(q)$.  Then the only difference between $\tilde R_{u,v}^J(q)$ and $q\tilde R_{u,sv}^J(q)$ is a single factor of $q$, yielding the result.
\end{proof}

\section{The case $s\in D_L(v)\cap D_L(u)$}\label{restriction}
\subsection{A remark on markings}
Suppose that $s\in D_L(v)$, and that $u< v$.  Let $\Lambda$ be the diagram of $v$, with $\Lambda'$ the diagram for $sv$, and $M$ the diagram for $u$, with $M'$ the diagram for $su$ if $su\in W^J$.  Suppose furthermore that $W^J$ is not $E_6/D_5$ or $E_6$.  

If $s\in D_L(u)$, then there is an $s$-labeled outer corner in $M$.  Then the skew diagram $\Lambda\bs M$ differs from $\Lambda'\bs M'$ in the diagonal containing that outside corner, and possibly in the diagonals immediately above and below this diagonal.  The upper diagonal differs if $s'u>u$ and $s'u\in W^J$, where $s'$ is the label of the box immediately right of the $s$-labeled outside corner in the diagram for $u$,  and a similar condition holds for the lower diagonal.   

In almost the same way, if $s\in D_L(v)\bs D_L(u)$, and $su\in W^J$, the diagrams $\Lambda\bs M$, $\Lambda'\bs M$, and $\Lambda'\bs M'$ differ at most on the diagonal of $\Lambda\bs M$ containing the $s$-labeled outside corner, and on the diagonals immediately above and below it.

Thus the difference in the polynomials $\tilde R_{u,v}^J(q)$, $\tilde R_{u,sv}^J(q)$, and $\tilde R_{su,sv}^J(q)$ is determined by markings that differ on at most three diagonals.  We refer to this set of diagonals as $I$, and define $\tilde R_{u,v}^J(q)|_I$ to be the polynomial determined by the markings on the diagonals in $I$.  

\subsection{Another case of the recurrence}
With this remark in hand, we can take on the next case of the recursion for $R_{u,v}^J(q)$.

\begin{prop}
Suppose that $s\in D_L(u)\cap D_L(v)$.  Then $\tilde R_{u,v}^J(q)=\tilde R_{su,sv}^J(q)$.
\end{prop}
\begin{proof}
In this case, $su\leq u$, implying $su\in W^J$, so in particular the diagram for $su$ is contained within the diagram of $u$, and is in fact obtained by deleting the $s$-labeled outside corner.  Let $\Lambda_v$, $M_u$, and $M_{su}$ are the diagrams associated to $v,u$, and $su$, respectively.  Analyzing the diagrams of Section \ref{HSDiag}, we see the following marked diagrams.  Note that we draw only the three relevant diagonals, and do not cover the cases $E_6/D_5$ or $E_7/E_6$, or include diagrams symmetric to given diagrams.  In each pair of diagrams, the left one is taken from $\Lambda\bs M$, and the right one is from $\Lambda'\bs M'$, and it is a matter of definitions to see that $\tilde r_{u,v}^J(q)|_I=\tilde r_{su,sv}^J(q)|_I$.

To be explicit, consider the restricted diagram marked as below:
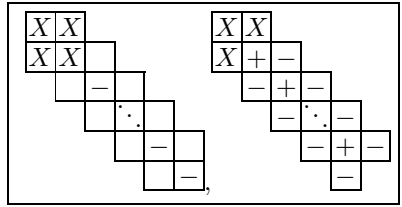
\[\framebox{
\young(XX,XX\ls,:\ls -\ls,::\ls\dlps\ls,:::\ls-\ls,::::\ls-),\young(XX,X+-,:-+-,::-\dlps-,:::-+-,::::-)}\]

Assume that the length of the $-$ marked diagonal in the first diagram is $k$; then $\tilde r_{u,v}^J(q)=\dfrac{1}{(q-1)^k\qb{k!}}$, while $\tilde r_{su,sv}^J(q)=\dfrac{(q-1)^{k}\qb{k!}}{(q-1)^{2k}\qb{k!}\qb{k!}}$, showing the two are equal.  As the skew diagram for the pair $su<sv$ contains exactly as many boxes as the diagram for $u<v$, $q^{\eta(u,v)}=q^{\eta(su,sv)}$, showing $\tilde R_{u,v}^J(q)=\tilde R_{su,sv}^J(q)$ in this case.

All other cases are similar; the relevant diagrams appear below:

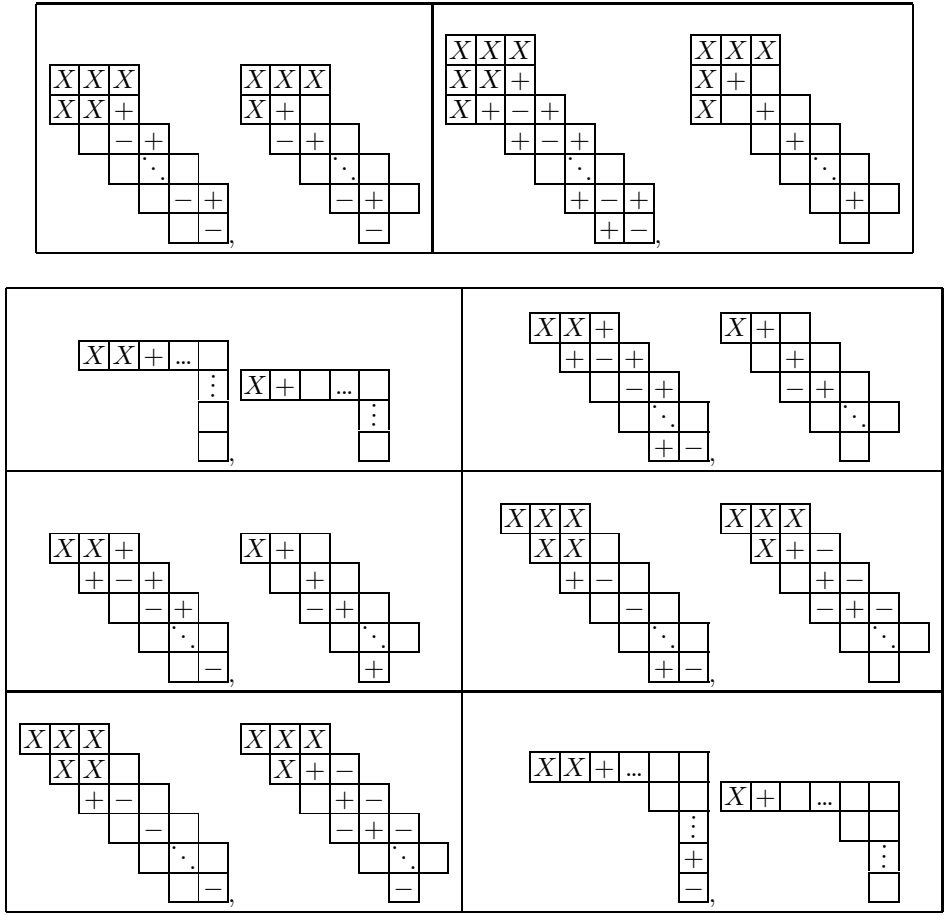
\[\begin{array}{|c|c|}\hline
\young(XXX,XX+,:\ls-+,::\ls\dlps\ls,:::\ls-+,::::\ls-),\young(XXX,X+\ls,:-+\ls,::\ls\dlps\ls,:::-+\ls,::::-)&
\young(XXX,XX+,X+-+,::+-+,:::\ls\dlps\ls,::::+-+,:::::+-),\raisebox{1.25in}{\ls}\young(XXX,X+\ls,X\ls+\ls,::\ls+\ls,:::\ls\dlps\ls,::::\ls+\ls,:::::\ls)\\\hline
\end{array}\]

\[\begin{array}{|c|c|}\hline
\young(XX+\lps\ls,::::\vlps,::::\ls,::::\ls),\young(X+\ls\lps\ls,::::\vlps,::::\ls)&
\raisebox{.9in}{\ls}\young(XX+,:+-+,::\ls-+,:::\ls\dlps\ls,::::+-),\young(X+\ls,:\ls+\ls,::-+\ls,:::\ls\dlps\ls,::::\ls)\\\hline
\young(XX+,:+-+,::\ls-+,:::\ls\dlps\ls,::::\ls-),\young(X+\ls,:\ls+\ls,::-+\ls,:::\ls\dlps\ls,::::+)&
\raisebox{1.1in}{\ls}\young(XXX,:XX\ls,::+-\ls,:::\ls-\ls,::::\ls\dlps\ls,:::::+-),\young(XXX,:X+-,::\ls+-,:::-+-,::::\ls\dlps\ls,:::::\ls)\\\hline
\young(XXX,:XX\ls,::+-\ls,:::\ls-\ls,::::\ls\dlps\ls,:::::\ls-),\young(XXX,:X+-,::\ls+-,:::-+-,::::\ls\dlps\ls,:::::-)&
\raisebox{1.1in}{\ls}\young(XX+\lps\ls\ls,::::\ls\ls,:::::\vlps,:::::+,:::::-),\young(X+\ls\lps\ls\ls,::::\ls\ls,:::::\vlps,:::::\ls)\\\hline
\end{array}\]

By the discussion above, this shows that if $u<v$ in $W^J$, and $s\in D_L(v)\cap D_L(u)$, then $\tilde R_{u,v}^J(q)=\tilde R_{su,sv}^J(q)$ for $W^J\notin\{E_6/D_5, E_7/E_6\}$.   The cases $E_6/D_5$ and $E_7/E_6$ can be verified using the computer. 
\end{proof}

\section{The case $u<su$, $su\in W^J$}
\subsection{Strategy, and a reduction}
As noted above, this proof will proceed by noting the possible restricted diagrams that may appear with an $s$-labeled outside corner of $v$, and an $s$ labeled inside corner of $u$.  In each case, we provide the diagrams associated to $u<v$, $u<sv$, and $su<sv$, respectively, and reference the part of Lemma \ref{q} (below) that is used in showing that $\tilde R_{u,v}^J(q)|_I=(q-1)\tilde R_{u,sv}^J(q)|_I+q\tilde R_{su,sv}^J(q)|_I$.  

Let $\Lambda_x$ denote the diagram of $x$ in $W^J$, and $\partial_{u,v}(I)$ the number of boxes in the restriction to $I$ of $\Lambda_v\bs\Lambda_u$.  Then $\partial_{u,v}(I)=\partial_{u,sv}(I)+1=\partial_{su,sv}(I)+2$.  Note that $\tilde R_{u,v}^J(q)|_I=\tilde r_{u,v}^J(q)q^g$, where $g$ is $\partial_{u,v}(I)-\deg(\tilde r_{u,v}^J(q)|_I)$.   In particular, if we wish to show that $\tilde R_{u,v}^J(q)|_I=(q-1)\tilde R_{u,sv}^J(q)+q\tilde R_{su,sv}^J(q)$, we can instead show $\tilde r_{u,v}^J(q)|_I=q^a(q-1)\tilde r_{u,sv}^J(q)|_I+q^b\cdot q\tilde r_{su,sv}^J(q)|_I$, where $a+\deg(\tilde r_{u,sv}^J(q)|_I)+1=\deg(\tilde r_{u,v}^J(q)|_I)$, and similarly for $b$.  Then multiplication by $q^g$, where $g$ is as above, generates the desired identity.  Calculating $g$ introduces extra notation, and can be avoided in the proof that follows.  Instead, we will multiply the terms coming from the diagrams for $u<sv$ and $su<sv$ by the power of $q$ to bring those terms to degree $\partial_{u,v}(I)-1$, resp. $\partial_{u,v}(I)-2$.

\subsection{A simple lemma}
We require a lemma to make the proof.
\begin{lem}\label{q}
For and $k\geq 1$,
\begin{enumerate}
\item $\qb{k}=\qb{k-1}+q^{k-1}$
\item $\qb{k}=1+q\qb{k-1}$ 
\end{enumerate}
\end{lem}

\begin{proof}
The result follows from the definition of quantum integers.
\end{proof}

\subsection{The final case of the recurrence}
To finish the proof of the main theorem, consider first the restricted diagram 
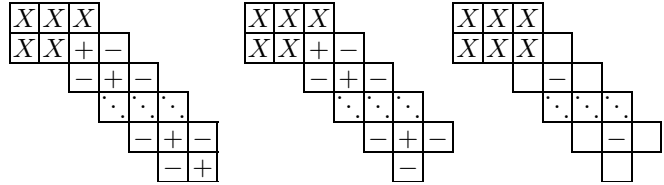
\[\young(XXX,XX+-,::-+-,:::\dlps\dlps\dlps,::::-+-,:::::-+)\quad\young(XXX,XX+-,::-+-,:::\dlps\dlps\dlps,::::-+-,:::::-)\young(XXX,XXX\ls,::\ls-\ls,:::\dlps\dlps\dlps,::::\ls-\ls,:::::\ls)\]
with outside corners north of and west of the inside $s$-labeled corner, occurs in the Type $A_n$ quotients, as well as $C_n/A_{n-1}, D_n/A_{n-1}$, $E_6/D_5$, and $E_7/E_6$, with $k$ $+$ marked boxes.  This diagram yields
\[\tilde r_{u,v}^J(q)|_I=(q-1)^{2-k}\frac{\qb{k}}{\qb{(k-1)!}}\]
which is of degree $1-\frac{(k-2)(k-1)}{2}$.  Then $\tilde r_{u,sv}^J(q)|_I=(q-1)^{1-k}\dfrac{1}{\qb{(k-1)!}}$, which has degree $1-k-\frac{(k-2)(k-1)}{2}$, differs from expectations by $1-k$, and $\tilde r_{su,sv}^J(q)=(q-1)^{2-k}\frac{1}{(k-2)!}$ with degree $2-k-\frac{(k-3)(k-2)}{2}$, differs from the expected degree by $1$.  Thus the relevant identity to show is
\begin{eqnarray*}(q-1)^{2-k}\frac{\qb{k}}{\qb{(k-1)!}}&=&(q-1)q^{k-1}\cdot \dfrac{(q-1)^{1-k}}{\qb{(k-1)!}}+q q^{-1}\cdot \frac{(q-1)^{2-k}}{(k-2)!}\\
&=&\frac{(q-1)^{2-k}}{\qb{(k-1)!}}\left( q^{k-1}+\qb{k-1}\right) \end{eqnarray*}
which is Lemma \ref{q} (1).

From here, we describe on, we cease to explicitly calculate the expected degrees.  In all cases, $k$ is the number of $+$ signs in the first diagram of a triple.

For our next diagram, take
\[\young(XXX,:X+\ls,::-+\ls,:::\dlps\dlps\dlps,::::-+)\qquad\young(XXX,:X+\ls,::-+\ls,:::\dlps\dlps\dlps,::::-)\qquad \young(XXX,:XX+,::\ls-+,:::\dlps\dlps\dlps,::::\ls)\]
corresponds to the identity
\[(q-1)\qb{k}=(q-1)q^{k-1}+q\cdot q^{-1}(q-1)\qb{k-1}\]
which reduces to Lemma \ref{q}, (1).

This appears in the Type $A_n$ quotients, as well as $C_n/A_{n-1}, D_n/A_{n-1}$, $E_6/D_5$, and $E_7/E_6$.

Next consider
\[\young(XXX,X+\ls,:\ls+\ls,::\dlps\dlps\dlps,:::\ls+)\qquad \young(XXX,X+\ls,:\ls+\ls,::\dlps\dlps\dlps,:::\ls)\qquad \young(XXX,XX+,:+-+,::\dlps\dlps\dlps,:::+)\]
corresponds to the identity
\[(q-1)^k\qb{k!}=(q-1)\qb{(k-1)!}q^{k-1}+q(q-1)^k\qb{k-1}\qb{(k-1)!}q^{-1}\]
which reduces to Lemma \ref{q}, (1).

This appears in the Type $A_n$ quotients, as well as $C_n/A_{n-1}, D_n/A_{n-1}$, $E_6/D_5$, and $E_7/E_6$.

We now move to restricted diagrams that appear in $D_n/A_{n-1}$ and $C_n/A_{n-1}$.  Consider the restricted diagram of $\Lambda\bs M$
\[\young(XXX,:ac,::bc,:::ac,::::\dlps\dlps,:::::a)\]
where $a$ is $s_n$, and $b$ is $s_{n-1}$, or vice versa.  Then we obtain the following triple of marked diagrams
\[\young(XX,:+\ls,::\ls\ls,:::+\ls,::::\dlps\dlps,:::::\ls\ls,::::::+)\qquad\young(XX,:+\ls,::\ls\ls,:::+\ls,::::\dlps\dlps,:::::\ls\ls)\qquad\young(XX,:X+,::\ls+,:::-+,::::\dlps\dlps,:::::\ls+)\]
(note the first box in the bottom row of the last diagram cannot be marked, by reason of parity).  Suppose there are $k$ $+$ marked boxes in the first diagram.  Then this restricted diagram corresponds to the identity
\[(q-1)^k\qb{1}\qb{3}\cdots\qb{2k-1}=(q-1)^{k}q^{2k-2}\qb{1}\qb{3}\cdots\qb{2k-3}+(q-1)^k\qb{1}\qb{3}\cdots\qb{2k-3}\qb{2k-2}\]
which reduces to Lemma \ref{q}, (1).

The restricted diagram of $\Lambda\bs M$
\[\young(XXX,:ab,::ab,:::ab,::::\dlps\dlps,:::::a)\]
with $a=s_n$, $b=s_{n-1}$ with an odd number of copies of $s_n$ appears as a restricted interval in $C_n/C_{n-1}$, and by our labeling algorithm, bears the same marks as the above case, and so satisfies the same identity.  

The same restricted diagram, but with an even number of $s_n$ labeled boxes, appears in $C_n/C_{n-1}$ but not $D_n/D_{n-1}$ (as we assume in that the label in the lowest box $b$ is a descent of $v$, which in $D_n$ requires $b$ to be an odd box).  Here the identity is
\[(q-1)^k\qb{1}\qb{3}\cdots\qb{2k-1}=q^{-1}\left((q-1)^{k+1}\qb{1}\qb{3}\cdots\qb{2k-1}\right)+q^{-2}\left(q(q-1)^k\qb{1}\qb{3}\cdots\qb{2k-1}\right)\]
which is true by inspection.

Consider the diagram $\Lambda\bs M$ corresponding to $u<v$
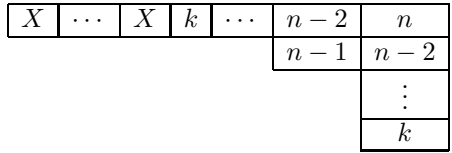
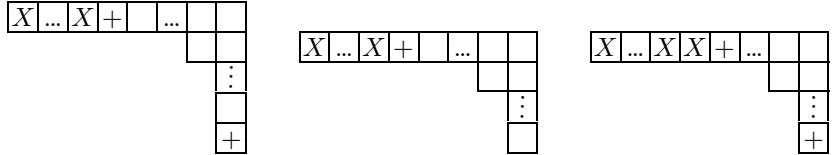
\[\begin{array}{|c|c|c|c|c|c|c|}\hline
X&\cdots& X&k&\cdots& n-2&n\\\hline
\multicolumn{5}{c|}{}&n-1&n-2\\\cline{6-7}
\multicolumn{6}{c|}{}&\vdots\\\cline{7-7}
\multicolumn{6}{c|}{}&k\\\cline{7-7}\end{array}\]
with marked triple
\[\young(X\lps X+\ls\lps\ls\ls,::::::\ls\ls,:::::::\vlps,:::::::\ls,:::::::+)\qquad\young(X\lps X+\ls\lps\ls\ls,::::::\ls\ls,:::::::\vlps,:::::::\ls)\qquad \young(X\lps XX+\lps\ls\ls,::::::\ls\ls,:::::::\vlps,:::::::+)\]
and note if two boxes are $+$ marked, their weights are $1$ and $k$, where the $+$ marked box is labeled $s_k$.  From this, we can read $\tilde R_{u,v}^J(q)=(q-1)^2\qb{k}q^{k-1}$, while $\tilde R_{u,v}^J(q)$ for the second diagram is $\tilde R_{u,sv}^J(q)=(q-1)q^{2k-1}$, and the third is $\tilde R_{su,sv}^J(q)=(q-1)^2\qb{k-1}q^{k-2}$.  Substituting into the recursion, we can reduce to the identity in Lemma \ref{q}, (2).  

This diagram appears in $D_n/D_{n-1}$, and can be renumbered
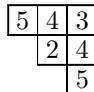
\[\young(543,:24,::5)\]
to be found in $E_6/D_5$ and $E_7/E_6$.

In $B_n/B_{n-1}$, the only relevant diagrams have the following shape
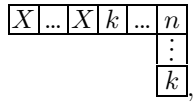
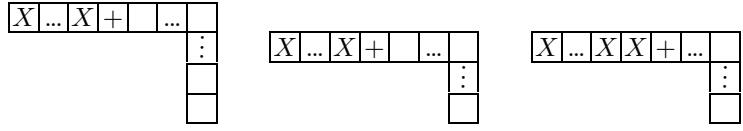
\[\young(X\lps Xk\lps n,:::::\vlps,:::::k),\]
with marked triple
\[\young(X\lps X+\ls\lps\ls,::::::\vlps,::::::\ls,::::::\ls)\qquad \young(X\lps X+\ls\lps\ls,::::::\vlps,::::::\ls)\qquad \young(X\lps XX+\lps\ls,::::::\vlps,::::::\ls)\]
We show the full diagram, but this is only to aid comprehension.  The usual three diagonals are all that is important to the calculation.

The relevant identity is
\[(q-1)q^k=(q-1)\cdot (q-1)q^{k-1}+q\cdot (q-1)q^{k-2}\]
which is trivially verified.

As an aside, it is straightforward to use the above identity, and induction, to show the following.
\begin{prop}\label{BnClosed}
If $[u,v]$ is an interval in $B_n/B_{n-1}$, then $R_{u,v}^J(q)=(q-1)q^{\ell(v)-\ell(u)-1}$.
\end{prop}

This type of interval appears in $B_n/B_{n-1}$, $C_n/C_{n-1}$, $E_6/D_5$, and $E_7/E_6$.

The remaining cases are similar, and we give only the diagrams which have yet to appear:

In $C_n/C_{n-1}$ and $D_n/D_{n-1}$ we have
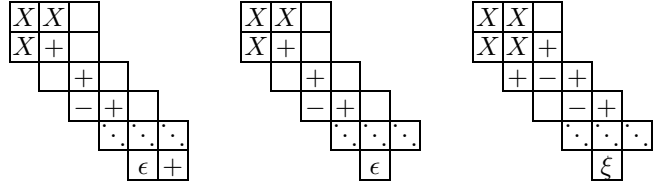
\[\young(XX\ls,X+\ls,:\ls+\ls,::-+\ls,:::\dlps\dlps\dlps,::::\epsilon+)\qquad\young(XX\ls,X+\ls,:\ls+\ls,::-+\ls,:::\dlps\dlps\dlps,::::\epsilon)\qquad\young(XX\ls,XX+,:+-+,::\ls-+,:::\dlps\dlps\dlps,::::\xi)\]
where $\e$ is $-$ if there are an even number of $+$ signs in the first diagram, and is unmarked otherwise, and $\xi$ is $+$ if there are an odd number of $+$ signs in the first diagram, and unmarked otherwise.

Similarly, in $C_n/C_{n-1}$ and $D_n/D_{n-1}$, we have the restricted diagrams
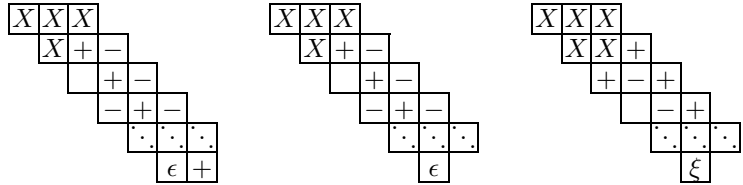
\[\young(XXX,:X+-,::\ls+-,:::-+-,::::\dlps\dlps\dlps,:::::\e+)\qquad\young(XXX,:X+-,::\ls+-,:::-+-,::::\dlps\dlps\dlps,:::::\e)\qquad \young(XXX,:XX+,::+-+,:::\ls-+,::::\dlps\dlps\dlps,:::::\xi)\]
with $\e$ and $\xi$ as above.

Finally, there are finitely many cases arising in $E_6/D_5$ and $E_7/E_6$ that may be obtained by hand or computer program.

\section{Combinatorial Invariance}\label{CombInv}
\subsection{The simplest cases}
As an application of the above algorithm, we show that the intersection cohomology of the cominiscule flag varieties is a combinatorial invariant, i.e. the cohomology is a function only of the poset.  Begin by recalling Proposition \ref{ordertheorem}, which provides a characterization of all covering relations in the Weyl quotient posets of Hermitian symmetric type. In particular, the poset $W^J$ is isomorphic to the containment order on partitions of the given types.  All that remains is to show is that a skew diagram uniquely determines an $R^J$ polynomial.
\begin{prop}
If a skew diagram $v\bs u$ consists of a single row of length $k$, it can be identified with the polynomial $(q-1)q^{k-1}$.  The same statement is true for columns.
\end{prop}
\begin{proof}
The leftmost box in $v\bs u$ is easily observed to be an ascent of $u$ within $v$ for any diagram of Hermitian type.  No other box is an ascent, or descent, of $u$, thus the polynomial is as stated.  The proof for columns is almost the same.
\end{proof}

Next, we note that if two restricted diagrams from Type $A$ have the same markings, then they have the same $R$-polynomials.

\begin{prop}
If $u<v$ in Type $A$, with shapes $\l$ and $\mu$ for $v$ and $u$ respectively, then $R_{u,v}^J(q)$ depends only on the skew shape $\l\bs\mu$.
\end{prop}
\begin{proof}
In the proof of the main theorem for Type $A$ above, only the length of certain diagonals was used to construct $R_{u,v}^J(q)$, and the diagonals are determined uniquely by the skew shape $v\bs u$.  
\end{proof}

\subsection{Combinatorial invariance}
Recall that we say that a diagram without its labels is a partition. In particular, note that a diagram $\Lambda$ is associated with a class of partitions, notably all such partitions that can contain that shape with that filling.  Meanwhile the partition $\l$ associated $\Lambda$ is associated with that class, and any other diagram $M$ that, when stripped of its filling, has shape $\l$.  We will use upper case Greek letters for diagrams, and lower case Greek letters for partitions in this section.  Recall in particular the definitions of Section \ref{Sdiagrams}.

\begin{prop}
If $\Lambda\bs M$ is a standard skew diagram of type $W^J$, it can be identified with a skew diagram in type $A$, and the $R$-polynomial associated to $\Lambda\bs M$ in $W^J$ is the same as the $R$-polynomial of the same shape in type $A$.
\end{prop}
\begin{proof}
It is easy to see in the diagrams of Section \ref{HSDiag} that a standard tableau of any type other than $C_n/A_{n-1}$ is supported on a set of simple roots whose Coxeter graph is an induced subgraph of a Coxeter graph of type $A$.

For the $C_n/A_{n-1}$ case, the standard skew diagram for $v\bs u$ can include a single box labeled $s_n$ in the bottom left corner (see figure below).  If this corner is not an inside corner, then the box North and West of it (in the diagram for $u$) is an outside corner, and by the rules established for $C_n/A_{n-1}$, the box associated to $s_n$ is unmarked, as it would be in type $A$. If it is an inside corner, it is marked with a plus, exactly as it would be in Type $A$.  All other boxes in the diagram are marked according to type $A$ rules, and thus the polynomial associated to $\Lambda\bs M$ in the context of the quotient $C_n/A_{n-1}$ is the same as the polynomial associated with the same diagram in type $A$.
\end{proof}
\begin{eqnarray*}
\young(X321,:432)=\young(X+\ls\ls,:\ls+\ls)&\quad& \young(XX21,:432)=\young(XX+\ls,:+-+)\\
\text{$C_n/A_{n-1}$ standard, $s_n$ not inside}&\quad&\text{$C_n/A_{n-1}$ standard, $s_n$ inside}
\end{eqnarray*}

We call a partition a {\it staircase } if it can be identified with the partition underlying a Hermitian diagram of type $D_n/A_{n-1}$.  Note that the full list of maximal staircase partitions not explicitly belonging to $D_n/A_{n-1}$ or $C_n/A_{n-1}$ are 
\[\begin{array}{ccc}
\begin{array}{|c|c|}
\hline n-1&n\\\hline\multicolumn{1}{c|}{}&n-1\\\cline{2-2}\end{array}\in B_n/B_{n-1}&
\begin{array}{|c|c|c|}\hline5&4&3\\\hline\multicolumn{1}{c|}{}&2&4\\\cline{2-3}\multicolumn{2}{c|}{}&5\\\cline{3-3}\end{array}\in E_6/D_5\text{ or }E_7/E_6
&\begin{array}{|c|c|c|}\hline5&4&2\\\hline\multicolumn{1}{c|}{}&3&4\\\cline{2-3}\multicolumn{2}{c|}{}&5\\\cline{3-3}\end{array}\in E_7/E_6
\end{array}.\]
But in all cases, the rules are those of $D_n/A_{n-1}$, and so we can identify polynomials again as before.

The final general form of skew partition is an L shape, of which there are two exceptional maximal examples, and one infinite class:
\[\begin{array}{cc}
\begin{array}{|c|c|}\hline n-1&n\\\hline\multicolumn{1}{c|}{}&n-1\\\cline{2-2}\multicolumn{1}{c|}{}&\vdots\\\cline{2-2}\multicolumn{1}{c|}{}&1\\\cline{2-2}\end{array}\in B_n/B_{n-1}
&\begin{array}{|c|c|}\hline 3&4\\\hline\multicolumn{1}{c|}{}&5\\\cline{2-2}\multicolumn{1}{c|}{}&6\\\cline{2-2}\multicolumn{1}{c|}{}&7\\\cline{2-2}\end{array}\in E_7/E_6\\
\begin{array}{|c|c|c|c|}\hline 7&6&5&4\\\hline\multicolumn{3}{c|}{}&2\\\cline{4-4}\end{array}\in E_7/E_6
&\begin{array}{|c|c|}\hline 2&4\\\hline\multicolumn{1}{c|}{}&5\\\cline{2-2}\multicolumn{1}{c|}{}&6\\\cline{2-2}\multicolumn{1}{c|}{}&7\\\cline{2-2}\end{array}\in E_7/E_6\text{ and restricted in }E_6/D_5
\end{array}.\]
In type $B_n/B_{n-1}$, the have $R^J_{u,v}(q)=(q-1)q^{\ell(v)-\ell(u)-1}$ via Proposition \ref{BnClosed}, while the rule in type $E$ is clear: one plus in the upper left box, and no other signs.  Thus the relative $R$-polynomials are the same.

The final cases to compare are intervals contained solely in $E_6/D_5$ and $E_7/E_6$, which requires a computer check.  Thus we have:

\begin{thm}\label{CombInvR}
If $\l\bs\mu$ is any skew partition, and $\Lambda\bs M$, $\Lambda'\bs M'$ are skew diagrams of shape $\l\bs\mu$, from types $W^J$ and $(W')^{J'}$ respectively, then $R^J_{u,v}(q)=R^{J'}_{u',v'}(q)$.
\end{thm}

\end{document}